\documentclass{amsart}
 
\usepackage{colordvi,multicol,color}

\usepackage{enumerate}
\usepackage{multicol}
\newtheorem{theorem}{Theorem}[section]
\newtheorem{lemma}[theorem]{Lemma}
\newtheorem{corollary}[theorem]{Corollary}
\theoremstyle{definition}
\newtheorem{definition}[theorem]{Definition}
\newtheorem{example}[theorem]{Example}

\newtheorem{proposition}[theorem]{Proposition}

\theoremstyle{remark}
\newtheorem{remark}[theorem]{Remark}
\newtheorem{notation}[theorem]{Notation}

\numberwithin{equation}{section}

\begin{document}

\title{On simple Lie $2$-algebra of toral rank $3$.}

\author{Carlos R. Payares Guevara}
\address{Facultad de ciencias B\'asicas,  Universidad Tecnol\'ogica de  Bolivar, Cartagena de Indias - Colombia}
\email{cpayares@utb.edu.co}

\author{Jeovanny de J. Muentes Acevedo}
\address{Facultad de ciencias B\'asicas,  Universidad Tecnol\'ogica de  Bolivar, Cartagena de Indias - Colombia}
\email{jmuentes@utb.edu.co}



\date{February 7, 2019}



\begin{abstract}
Simple Lie algebras  of finite dimension over an algebraically closed field of characteristic 0 or $p> 3$ were recently classified. However, the problem  over  an algebraically closed field of characteristics 2 or 3 there exist only partial results. The first result on the problem of classification of simple Lie algebra of finite dimension over an algebraically closed field of characteristic 2 is that these algebras have  absolute toral rank greater than or equal to 2. 
In this paper we show that there are not simple Lie $2$-algebras with toral rank $3$ over an algebraically closed field of characteristic $2$ and dimension less or equal to $16$.
\end{abstract}



\maketitle

\section*{Introduction}
The simple Lie algebras of  finite  dimension over an algebraically closed field of characteristic zero were first classified by Killing ($1888$) and
Cartan ($1894$).  These algebras fall in four infinite families, $A_{n}$, $B_{n}$, $C_{n}$ and $D_{n}$, and five exceptional cases, $E_{6}$, $E_{7}$, $E_{8}$, $G_{2}$ and $F_{4}$ (see \cite{ja}).

\medskip

   Simple finite-dimensional Lie algebras over an algebraically closed field of prime characteristic $ p> 7 $ were classified by  H. Strade, R. Block and R. L. Wilson in the middle of years 90 (see \cite{BL1}, \cite{BL2}, \cite{BL3}, \cite{ST1},  \cite{ST3}, \cite{ST4}, \cite{ST5}, \cite{ST6} and \cite{Wi}). In a series of papers,  H. Strade and A. Premet classified the finite-dimensional simple Lie algebras over an algebraically closed field of characteristic $ p = 5 $ and $p = 7$ in the beginning
of this century  (see \cite{PS1}, \cite{PS2}, \cite{PS3}, \cite{PS4}, \cite{PS5} and \cite{PS6}).  It asserts that every simple Lie algebra over an
algebraically closed field of characteristic $p> 3$ is either classical, or of Cartan, or Melikian type.



\medskip

After the classification of simple Lie algebras, of finite dimension, over a field of  characteristic $p> 3$, the main problem still open in the category of Lie algebras of finite dimension is the classification of simple Lie algebras on
 an algebraically closed field of characteristic $p = 2$ and $p = 3$.
  In particular for $p = 2$   many new phenomena arise    (for instances, 
simple Lie algebras in   characteristic zero are not necessarily simple in characteristic two) and
the classification will differ significantly from those in characteristic $0$ and $p >3$. The first results for the classification problem in characteristic $2$ were made by S. Skryabin in \cite{gr21}. In that work,  S. Skryabin proved that all  finite dimensional Lie algebra of absolute toral rank $1$ over an algebraically closed field $K$ of characteristic $2$ is solvable, or  equivalently, all finite dimensional simple Lie algebra over an algebraically closed field of characteristic $2$ has absolute toral rank of at least $2$.

\medskip

In \cite{gr1}, Section 6, A. Premet and H. Strade present the following problem which is wide open.

\noindent \textbf{Problem 1}. \textit{Classify all finite dimensional simple Lie algebras of absolute toral
rank two over an algebraically closed field of characteristics 2 and 3}.  

\medskip

Strong results closely related to this problem were obtained by A. Grishkov and A. Premet in \cite{gr} (work in progress). They annouced the following result:
All  finite dimensional simple Lie algebra over an algebraically closed field of characteristic $2$ of absolute toral rank $2$ are classical of dimesion $3$, $8$, $14$ or $26$. In particular, all  finite dimensional simple Lie $2$-algebra over a field of characteristic $2$ of (relative) toral rank $2$ is isomorphic to $A_{2}$, $G_{2}$ or $D_{4}.$

\medskip

The case when the absolute toral rank or relative toral rank  is greater than or equal to $3$   is much more difficult.
For these  cases, we propose the followings problems, which are still open: 

\medskip

\noindent \textbf{Problem 2}. \textit{Classify all finite dimensional simple Lie algebras of absolute toral
rank greater than or equal to $3$ over an algebraically closed field of characteristic $2$.}

\medskip

\noindent \textbf{Problem 3}. \textit{Classify all finite dimensional simple Lie $2$-algebras of relative toral
rank greater than or equal to $3$ over an algebraically closed field of characteristic $2$.}

\medskip

In this paper we will give a partial result for Problem 3 when the relative toral rank is $3$. Specifically,  we show that there are not simple Lie $2$-algebras of dimension less than or equal to $16$ with  (relative) toral rank $3$ over an algebraically closed field of characteristic $2$.

\medskip

In the next section we will present some basic definitions and well-known results that will be used throughout the work. In Section 2, we will prove that there are no simple Lie 2-algebras with toral rank 3 and of dimension less than or equal to 6. In Section 3, we find the Cartan  decomposition of a Lie 2-algebra of dimension greater than or equal to 7 with respect to a maximal subalgebra of dimension 3. 
This decomposition will give us the necessary tools for the classification of these algebras and  will be used for the rest of the work. 
The study of the Cartan decomposition  of these algebras with respect to a toral subalgebra of dimension 3 when the root spaces have cardinality less than 7 will allow us to conclude in Section 4 that there are no simple Lie 2-algebras of toral rank 3 with dimension between 7 and 9. When the root space is equal to the dual of the toral subalgebra, the dimension of the algebra is greater than or equal to 10. Using this fact and Theorem \ref{theorem53}, in which we  classify the algebras in whose decomposition of Cartan the root  spaces   have different dimensions,   we will prove in the last section that there are no simple Lie 2-algebra of dimension between 10 and 16.




\section{Preliminaries}
Throughout this paper all the  algebras are finite-dimensional and are defined over a fixed algebraically closed field $K$ of characteristic $2$ containing the prime field  $\mathbb{F}_{2}$. We will start presenting  some basic definitions and well-known facts.

\begin{definition}
A \textit{Lie $2$-algebra} is a pair $(\mathfrak{g}, [2])$, where $\mathfrak{g}$ is a Lie algebra over $K$, and $[2]: \mathfrak{g}\rightarrow  \mathfrak{g} $, $ a\mapsto a ^ {[2]} $ is a map (called $2$-map) such that:
\begin{enumerate}
 \item  $ {(a + b)} ^ {[2]} = {a} ^ {[2]} + {b} ^ {[2]} + [a, b] $,  $\forall a, b\in\mathfrak{g} $.
 \item  $ (\lambda a) ^ {[2]} = \lambda ^ {2} a ^ {[2]} $,  $ \forall \lambda \in K $, $ \forall a \in \mathfrak{g}. $
 \item $ \text{ad} ({b} ^ {[2]}) = (\text{ad} (b)) ^ {2} $, $ \forall  b \in \mathfrak{g} $.
 \end{enumerate}
 \end{definition}

If the $2$-map $[2]$ exists, it is unique for 
any Lie $2$-algebra $(\mathfrak{g}, [2])$ with $\mathfrak{z} (\mathfrak{g}) = 0$. A  Lie $2$-algebra 
 $(\mathfrak{g}, [2])$ is called   \textit{simple} if   $\mathfrak{g}$ is a simple Lie algebra on $K$.

 \begin{example}
  Set \[\mathfrak{o}_{3}(K):=\{A=(a_{ij})\in M_3(K) :a_{ii}=0\text{ and }  a_{ij}=a_{ji} \}.\] The canonical basis of $\mathfrak{o}_{3}(K)$ is given by: $e_1 := e_{12}+e_{21}$, $e_2 :=  e_{13}+e_{31}$, and $e_3 := e_{23}+e_{32}.$  Then $\mathfrak{o}_{3}(K) = Ke_1\oplus Ke_2\oplus Ke_3,$ with $  [e_1,e_2]=e_3,\; [e_1,e_3]=e_2,\; [e_3,e_1]=e_2$, is the only (up to isomorphism) simple Lie algebra of dimension $3$ on $K$. However,  it is not a Lie $2$-algebra. 
 \end{example}

\begin{definition} Let $(\mathfrak{g},[2])$ be a  Lie
$2$-algebra over $K$. An element $x\in\mathfrak{g}$ is called \textit{semisimple} (respectively, \textit{$2$-nilpotent}) if
$x$ lies in the $2$-subalgebra of $\mathfrak{g}$ generated by $x$, that is $x\in \text{span}\{x,{x}^{[2]},{x}^{[2]^{2}},...\}$
(respectively, if ${x}^{[2]^{n}}=0$ for some  $n\in \mathbb{N}$).\end{definition}

It is well-known that for any $x \in\mathfrak{g}$ there are unique elements  $x_{s}$ and $x_{n}$ in $\mathfrak{g}$ such that
$x_{s}$ is semisimple, $x_{n}$ is $2$-nilpotent, and  $x = x_{s} + x_{n}$ with $[x_{s},x_{n}]=0$ (Jordan-Chevalley-Seligman decomposition).

\medskip


\begin{definition}A \textit{torus} in $\mathfrak{g}$ is an abelian subalgebra $\mathfrak{t}$ for which the $2$-mapping is one-to-one.\end{definition}

For  a torus $\mathfrak{t}$ there is a basis $\{t_{1},... , t_{n}\}$ such that $t^{[2]}_{i} = t_{i}$. The elements satisfying  $t=t^{[2]}$ will be called \textit{toroidal elements}.
A torus $\mathfrak{t}_{1}$ of $\mathfrak{g}$ is called \textit{maximal} if the inclusion $\mathfrak{t_{1}}\subseteq \mathfrak{t_{2}}$, with $\mathfrak{t_{2}}$ toral,  implies $\mathfrak{t_{1}} =\mathfrak{ t_{2}}.$ 

\begin{definition}(H. Strade \cite{ST2}).
The \textit{(relative) toral rank} of a  Lie $2$-algebra $(\mathfrak{g},[2])$ is given by \[ {MT} (\mathfrak{g}):= \max \{\text{dim}_{K}(\mathfrak{t}): \mathfrak{t}\text{ is a torus in }  \mathfrak{g} \}. \]
\end{definition}

For instances, the centralizer $\mathfrak{c}_{\mathfrak{g}}(\mathfrak{t})$ of any maximal torus in $\mathfrak{g}$ is a Cartan subalgebra of $\mathfrak{g}$ and, conversely, the semisimple elements of any Cartan subalgebra of $\mathfrak{g}$ lie in its center
and form a maximal torus in $\mathfrak{g}$.


\medskip

Let $\mathfrak{t}$ be a maximal torus of $(\mathfrak{g}, [2])$,  $\mathfrak{h}:=\mathfrak{c}_{\mathfrak{g}}(\mathfrak{t})$, and let $V$ be a finite dimen\-sional  $(\mathfrak{g},[2])$-module (this means that $\rho_{V}(x^{[2]})=\rho_{V}(x)^{2}$,  $\forall\, x \in \mathfrak{g}$ where $\rho_{V}\colon \mathfrak{g}\to  \text{End}_{K}(V)$  denotes the corresponding $2$-representation). Since $\mathfrak{t}$ is abelian, then $\rho_{V}(\mathfrak{t})$ is abelian and consists of  semisimple  elements. Therefore,  $V$ can be decomposed into weight spaces respect to $\mathfrak{t}$ as
\[V=  \underset{\lambda\in {\mathfrak{t}}^{*}}\oplus  V_\lambda, \quad \text{ where } V_{\lambda}:=\{v\in V: \rho(t)(v):=v\cdot t = \lambda(t)\,v,\,\, \forall t\in \mathfrak{t}\}.\]

 The subset  $\Delta:=\{\lambda \in\mathfrak{t^{*}}: V_{\lambda}\neq 0\}\subseteq \mathfrak{t^{*}}$ is called the  $\mathfrak{t}$-\textit{weight} of $V$.  

\begin{lemma}\label{dvdg1} If $t$ is a toral element of $\mathfrak{t}$, then $\lambda(t)\in \mathbb{F}_{2}$, for any  $\lambda\in\Delta $.\end{lemma}
\begin{proof}If $  v\in V\setminus \{0\}$, then  \[ \lambda(t)v=v \cdot t=v\cdot t^{[2]}=(v\cdot t)\cdot t=(\lambda(t)v)\cdot t=\lambda(t)(v\cdot t)=\lambda(t)^2v.\] \, So,  $\lambda(t)\in\mathbb{F}_2$.
\end{proof}

If $V =\mathfrak{g}$ is the adjoint  $(\mathfrak{g},[2])$-module, then $\Delta=\{\lambda \in\mathfrak{t^{*}}: \mathfrak{g}_{\lambda}\neq 0\}$ is nothing but the set of roots of $\mathfrak{g}$ respect to $\mathfrak{t}$, and
$$ \mathfrak{g}=\mathfrak{h}\oplus\left(\underset{\alpha\in\Delta}\oplus \mathfrak{g}_\alpha\right) \quad\text{ where }\quad \mathfrak{g}_\alpha:=\{g\in \mathfrak{g}:\text{ad}(h)(g)=\alpha(h)g,\, h\in \mathfrak{h}\}$$ is the root space decomposition. 


\section{Simple Lie $2$-algebra with toral rank $3$ and $dim_{K}\leq 6$}

We will obtain in Proposition \ref{payares}  that there are no simple Lie $2$-algebras of $\text{dim}_{K}\leq 6$ with toral rank $3$. Next section will be spend to find the Cartan decomposition for    $\text{dim}_{K}\geq 7$, which will be very useful throughout the work. 

\medskip

Suppose that there exists a  simple Lie $2$-algebra $(\mathfrak{g}, [2])$ of $\text{dim}_{K}(\mathfrak{g})\leq 6$. In particular, $\mathfrak{g}$ is a simple Lie algebra on $K$ of $\text{dim}_{K}(\mathfrak{g}) \leq 6$ and $\mathfrak{g}$ is not isomorphic to the Witt algebra $W (1;\textbf{\underline{1}})$ (otherwise,  $W(1;\textbf{\underline{1}})$ is simple over $K$, which is absurd). Then, by \cite{kl}, Theorem 2.2, we have that $\mathfrak{g}$ is of dimension 3 and so $\mathfrak{g}$ is isomorphic to $\mathfrak{o}_{3}(K)$. However $\mathfrak{o}_{3}(K)$ is not a Lie $2$-algebra. This facts proves that:

\begin{proposition}\label{payares}
There are no simple Lie $2$-algebras of $\text{dim}_{K}\leq 6$ with toral rank $3$.
\end{proposition}

\section{Cartan decomposition of a Lie $2$-algebra  of rank toral $3$.}
As we said in Section 2, in this section we will suppose that   $ (\mathfrak{g}, [2]) $ is a   Lie $2$-algebra with $  {MT} (\mathfrak{g})=3 $ and we will find its Cartan decomposition.   We have $ \mathfrak{g} $ contains a maximal torus $ \mathfrak{t} $ of $ \text{dim}_{K} (\mathfrak{t}) = 3 $.  
Since $ \text{ad} (\mathfrak{t}) $ is abelian and consists of semisimple elements, $ \mathfrak{g} $ can be decomposed into  weights spaces respect to $ \mathfrak{t} $, that is
$$\mathfrak{g}=\mathfrak{c}_{\mathfrak{g}}(\mathfrak{t})\oplus \left(\underset{\lambda\in\ {\mathfrak{t}}^{*}\setminus\{0\}}\oplus \mathfrak{g}_\lambda\right) \quad \text{where}\quad  \mathfrak{g}_\lambda=\{v\in \mathfrak{g} : [t,v]= \lambda(t)v,\,\,\forall t\in \mathfrak{t}\}.$$
Now, since $ \mathfrak{t} $ is a torus of dimension 3, there is a basis $\{t_i\in \mathfrak{t}: t^{[2]}_i  = t_i, i = 1, 2, 3 \}$ of $\mathfrak{t}$ such that $\lambda(t_i) \in \mathbb{F}_{2}$,  $\forall\lambda\in \mathfrak{t}^{*}$ (Lemma \ref{dvdg1}). Hence, for any $\lambda:\mathfrak{t}\to \mathbb{F}_2 $,   writing 
$\lambda=(\lambda(t_{1}),\lambda(t_{2}),\lambda(t_{3}))$, 
we have  \[ \mathfrak{t}^{*}= \{(0,0,0),(1,0,0),(0,1,0),(0,0,1),(1,1,0),(1,0,1),(0,1,1),(1,1,1)\}.\]

Set $\alpha=(1,0,0) $, $\beta=(1,0,0)$, $\gamma= (0,0,1) $ and $\mathfrak{h}=\mathfrak{c}_ {\mathfrak{g}}(\mathfrak{t})$. Thus $\{\alpha, \beta, \gamma \} $ is a basis of $\mathfrak{t^{*}}$ and $\mathfrak{h}$ is a Cartan subalgebra of $\mathfrak{g}$. Thus, $\mathfrak{h} =\mathfrak{t}\oplus \mathfrak{n} $, where $ \mathfrak{n} $ is a $2$-nilpotent  subalgebra  of $\mathfrak{g} $ and $ [\mathfrak{t},\mathfrak{n}] = 0. $
Therefore:

\begin{theorem}\label{patron}
The decomposition into weights spaces of $ (\mathfrak{g},[2])$ with respect to $\mathfrak{t}$ is: \[\mathfrak{g}=\mathfrak{t} \oplus \mathfrak{n} \oplus \mathfrak{g}_{\alpha}\oplus\mathfrak{g}_{\beta} \oplus\mathfrak{g}_{\gamma}\oplus\mathfrak{g}_{\alpha + \beta}\oplus\mathfrak{g}_{\alpha+ \gamma}\oplus \mathfrak{g}_{\beta+\gamma}\oplus\mathfrak{g}_{\alpha + \beta + \gamma}\] or \[\mathfrak{g}=\mathfrak{t} \oplus \mathfrak{n}\oplus\left(\underset{\xi\in G}\oplus \mathfrak{g}_\xi\right),\] where $G:=\left\langle \alpha,\beta,\gamma\right\rangle$ is an elementary abelian group of order $8$ and $ \mathfrak{n} $ is a $2$-nilpotent  subalgebra  of $\mathfrak{g} $.
\end{theorem}


Using the decomposition in Theorem \ref{patron}, we show in the next proposition that there exist at least three non-zero weight spaces. This fact leads us to conclude that the dimension of $ \mathfrak{g}$ must be greater than or equal to 6, since the dimension of $  \mathfrak{h}$  is greater than or equal to 3.

\begin{proposition}\label{wefef}
There are  $\xi$, $\delta$, $\theta$ in $\mathfrak{t} ^{\ast}$ linearly independent such that $ \mathfrak{g}_\xi \neq 0 $, $ \mathfrak{g}_ \delta\neq0 $ and $ \mathfrak{g}_ \theta \neq0 .$
\end{proposition}
\begin{proof}
 Suppose the proposition is false. We have two cases:

\noindent \textbf{Case 1:}  There are $\xi$, $ \delta $ and $ \mathfrak{t}^{*} $ linearly independent such that $\mathfrak{g}_\xi \neq0 $ and $ \mathfrak{g}_ \delta \neq 0 $. 

In this case, we can  extend $\{\xi,\delta\}$ to a basis of $\mathfrak{t}^{*}$. Thus there is $\lambda\in \mathfrak{t}^{*}\backslash\{0\} $ such that $\{\xi, \delta,\lambda \}\subseteq \mathfrak{t}^{*}\backslash\{0 \}$ is a linearly independent set and $\mathfrak{g}_\lambda = 0 $. By the action of the linear group $ GL_{3}(\mathbb{F}_2)$ on $ \mathfrak{t}^{*}$ we can consider the following change of basis  given by
\[
 \begin{array}{rcl} 
{GL_{3}(\mathbb{F}_2)\times \mathfrak{t}^{*}} & \xrightarrow{\                                            
\psi_A\ } & \mathfrak{t}^{*}\\ 
 (A,\xi) & \mapsto & A \cdot \xi:=\alpha\\
(A,\delta) & \mapsto & A \cdot \delta:=\beta\\
(A,\lambda) & \mapsto & A \cdot \lambda:=\gamma\\
\end{array}
\]
Thus, we find that $\mathfrak{g}=\mathfrak{h}\oplus \mathfrak{g}_{\alpha} \oplus \mathfrak{g}_{\beta}\oplus \mathfrak{g}_{\alpha+\beta}.$ Consequently, up to change of basis, $ \mathfrak{g} $ can be written as  $$\mathfrak{g}=\mathfrak{h}\oplus \mathfrak{g}_{\varepsilon} \oplus \mathfrak{g}_{\mu}\oplus \mathfrak{g}_{\varepsilon+\mu},$$ where $\{\varepsilon,\mu,\mu+\varepsilon\}$
is a linearly dependent set of $ \mathfrak{t}^{*} $. Hence, $ \mathfrak{g} $ has  total range  $ 2 $, which is a contradiction.

\medskip

\noindent \textbf{Case 2:}  There is $ \xi \in \mathfrak{t}^{*}$ such that $\mathfrak{g}_ \xi\neq0 $.
 
We extend $\{\xi \}\subseteq \mathfrak{t}^{*}\backslash\{0 \} $ to a basis of $ \mathfrak{t}^{*}$, so for $\pi,\,\kappa\in\mathfrak{t}^{*}\setminus\{0 \} $,  $ \{\xi, \pi, \kappa\} $ is a linearly independent set such that $ \mathfrak{g}_{\xi} \neq 0 $ and $ \mathfrak{g}_{\pi}=\mathfrak{g}_{\kappa} = 0. $ Then, by the action of $ GL_3 (\mathbb{F}_2) $ on $ \mathfrak{t}^{*}$, we can make a change of basis, that is, for $A\in GL_3 (\mathbb{F}_2 ) $ fixed, we have
\[
 \begin{array}{rcl} 
{GL_3(\mathbb{F}_2)\times \mathfrak{t}^{*}} & \xrightarrow{\                                            
\psi_A\ } & \mathfrak{t}^{*}\\ 
(A,\xi) & \mapsto & A.\xi:=\alpha\\
(A,\pi) & \mapsto & A.\pi:=\beta\\
(A,\kappa) & \mapsto & A.\kappa:=\gamma.
\end{array}
\]

Therefore,  up to   change of basis, $\mathfrak{g}$ can be written as  $\mathfrak{g}=\mathfrak{h}\oplus \mathfrak{g}_{\varepsilon}$ for some $\varepsilon\in\mathfrak{t}^{*}\setminus\{0\} $. In this case, we have $\mathfrak{g}$ has toral rank $1$, which is a contradiction.
\end{proof}

It follows from Proposition \ref{wefef} that: 

\begin{corollary}\label{cor33} $\text{dim}_{K}(\mathfrak{g})\geq 6.$
\end{corollary}
\begin{proof}
Since $\text{dim}_{K}(\mathfrak{h})\geq 3$   and by Proposition \ref{wefef} there exist $\theta$, $\delta$ and $\xi$ in $\mathfrak{t}^{\ast}$ linearly independent such that $ \text{dim}_{K}(\mathfrak{g}_\xi)\geq 1 $, $ \text{dim}_{K}(\mathfrak{g}_ \delta)\geq1 $ and $\text{dim}_{K}( \mathfrak{g}_ \theta) \geq1 ,$ we have $\text{dim}_{K}(\mathfrak{g})\geq 6.$
\end{proof}

\begin{remark}\label{obs33} By Theorem \ref{patron},  Proposition \ref{wefef} and Corollary \ref{cor33},  from now on we will assume, without loss of generality, that $\text{dim}_{K}(\mathfrak{g})\geq 6$ and the Cartan decomposition into weight spaces of $\mathfrak{g}$ is  $$\mathfrak{g} =\mathfrak{t}\oplus \mathfrak{n} \oplus\mathfrak{g}_{\alpha}\oplus \mathfrak{g}_{\beta}\oplus\mathfrak{g}_{\gamma} \oplus\mathfrak{g}_{\alpha+\beta}\oplus \mathfrak{g}_{\alpha+\gamma}\oplus \mathfrak{g}_{\beta+\gamma}\oplus \mathfrak{g}_{\alpha+\beta+\gamma},$$
where $\{\alpha,\beta,\gamma\}$ is a basis of $\mathfrak{t}^{*}$ such that $\mathfrak{g}_{\alpha}\neq0$, $\mathfrak{g}_{\beta}\neq0$, and  $\mathfrak{g}_{\gamma}\neq0$,  with  $$\text{dim}_K(\mathfrak{g}_{\alpha})\geq \text{dim}_K(\mathfrak{g}_{\beta})\geq \text{dim}_K(\mathfrak{g}_{\gamma})\geq \text{dim}_K(\mathfrak{g}_{\xi})> 0, \quad \forall\,\xi\in \mathfrak{t}^{*}\backslash\{\alpha,\beta,\alpha+\beta\}.$$
\end{remark}

Let $ \Delta:= \{\lambda \in\mathfrak{t}^{*}: \mathfrak{g}_{\lambda}\neq0 \} $ be the root system of $ \mathfrak{g} $ with respect to $ \mathfrak{t} $. It follows from Theorem \ref{patron} that   $\alpha$, $\beta$ and $\gamma$  are elements of $\Delta $.  The following are all the  possibilities for   $ \Delta$ depending on its cardinality.  
\begin{enumerate}[1.]
\item If $\text{Card}(\Delta)=3$, we have $\Delta_1=\{\alpha,\beta,\gamma\}$.
\item If $\text{Card}(\Delta)=4$,  we have the following possibilities for $\Delta$: 
\begin{enumerate}[a.]
  \item $\Delta_{2}=\{\alpha,\beta,\gamma,\alpha+\beta\}$
  \item $\Delta_{3}=\{\alpha,\beta,\gamma,\alpha+\gamma\}$ 
  \item $\Delta_{4}=\{\alpha,\beta,\gamma,\beta+\gamma\}$
  \item $\Delta_{5}=\{\alpha,\beta,\gamma,\alpha+\beta+\gamma\}$
\end{enumerate}
\item If $\text{Card}(\Delta)=5$, we have the following possibilities for $\Delta$:
\begin{enumerate}[a.]
\item$\Delta_{6}=\{\alpha,\beta,\gamma,\alpha+\beta,\alpha+\gamma\}$ 
\item$\Delta_{7}=\{\alpha,\beta,\gamma,\alpha+\beta,\beta+\gamma\}$
\item$\Delta_{8}=\{\alpha,\beta,\gamma,\alpha+\beta,\alpha+\beta+\gamma\}$ 
\item$\Delta_{9}=\{\alpha,\beta,\gamma,\alpha+\gamma,\beta+\gamma\}$
\item$\Delta_{10}=\{\alpha,\beta,\gamma,\alpha+\gamma,\alpha+\beta+\gamma\}$
\item$\Delta_{11}=\{\alpha,\beta,\gamma,\beta+\gamma,\alpha+\beta+\gamma\}$ 
\end{enumerate}
\item If $\text{Card}(\Delta)=6$, we have the following possibilities for $\Delta$:
\begin{enumerate}[a.]
\item$\Delta_{12}=\{\alpha,\beta,\gamma,\alpha+\beta,\alpha+\gamma,\beta+\gamma\}$ 
\item$\Delta_{13}=\{\alpha,\beta,\gamma,\alpha+\beta,\alpha+\gamma,\alpha+\beta+\gamma\}$
\item$\Delta_{14}=\{\alpha,\beta,\gamma,\alpha+\beta,\beta+\gamma,\alpha+\beta+\gamma\}$
\item$\Delta_{15}=\{\alpha,\beta,\gamma,\alpha+\gamma,\beta+\gamma,\alpha+\beta+\gamma\}$ 
\end{enumerate}
\item If $\text{Card}(\Delta)=7,$ then $\Delta_0=\mathfrak{t}^*\backslash\{0\}$.
\end{enumerate}

For $0\leq i\leq 15,$ we will denote by $(\mathfrak{g}^{\Delta_{i}},[2])$ the Lie $2$-algebra with its associated root space $\Delta_{i}$ in the previous list. 

\medskip

 Using the above facts we can   conclude the following theorem:

  \begin{theorem}  For each  $0 \leq i \leq  15$, we have   the  Cartan decompositions of  $(\mathfrak{g}^{\Delta_{i}},[2])$ with respect to $ \mathfrak{h} $ is given by  \begin{equation}\label{jeo1} \mathfrak{g}^{\Delta_i}=\mathfrak{t} \oplus\mathfrak{n} \oplus\left(\underset{\xi\in\Delta_i}\oplus\mathfrak{g}_\xi
\right).\end{equation}\end{theorem}


\section{Analysis of $\mathfrak{g}=\mathfrak{h}\oplus (  \oplus_{\xi\in\Delta_{i}} \mathfrak{g}_{\xi})$, with $Card(\Delta_{i})<7$.}

In this section we will show that for each $ \Delta_i $, $ 1 \leq i\leq15 $, the Cartan decomposition 
$ \mathfrak{g}^{\Delta_i} =\mathfrak{h}\oplus(\underset{\xi\in\Delta_i}\oplus\mathfrak{g}_ \xi) $, associated with each system of root $\Delta_i $, generate some contradictions. These incompatibilities allow us to conclude that there is not a simple Lie $2$-algebra   of toral rank $3$ with these structures (see Proposition \ref{payarin}).  The case $\Delta_0=\mathfrak{t}^{*}\backslash\{0\}$ corresponding to $i=0$ will be studied  in the next section.

\medskip

 From now on we suppose that $(\mathfrak{g},[2])$ is a simple Lie $2$-algebra with $MT(\mathfrak{g})=3$ and $\text{dim}_{K}(\mathfrak{g})\geq 7$ (since there is no simple Lie $2$-algebra with $MT(\mathfrak{g})=3$ and $\text{dim}_{K}(\mathfrak{g})=6$). Furthermore, we will assume that $\mathfrak{h}$  is a Cartan subalgebra  of $\mathfrak{g}$ of maximal  rank  and, therefore, $\mathfrak{h}\oplus \mathfrak{g}_{\xi}$ is solvable for all $\xi\in G=\left\langle \alpha,\beta,\gamma\right\rangle,$ where $G=\left\langle \alpha,\beta,\gamma\right\rangle$ is the    elementary abelian group of order $8$.

\medskip

In order to prove Proposition \ref{payarin}, we show the next lemmas: 

\begin{lemma}\label{edfrr}
If  $\text{dim}_{K}(\mathfrak{g_{\xi}})=1$ for $\xi \in \Delta\setminus \{0\}$, then $[\mathfrak{n},\mathfrak{g_{\xi}}]=0.$
\end{lemma}
\begin{proof}
Set $\mathfrak{g_{\xi}}=\text{span}\{e_{\xi}\}$  and take  $n\in\mathfrak{n}$. Thus $$y=\text{ad}(n)(e_{\xi})=[n,e_{\xi}]\in[\mathfrak{n},\mathfrak{g}_{\xi}]\subset \mathfrak{g}_{\xi}.$$  Therefore $\text{ad}(n)(e_{\xi})=\lambda e_{\xi}$ for some $\lambda \in K .$ Since $\text{ad}(n)$ is a nilpotent operator, we have $\lambda=0$ and therefore $y=0,$ which proves the lemma. 
\end{proof}

\begin{lemma}\label{patrondepatrones}
If $\mathfrak{g}=\mathfrak{h}\oplus\left(\underset{\xi\in\Delta_i}\oplus \mathfrak{g}_{\xi}\right)$, then $[\mathfrak{g}_{\xi},\mathfrak{g}_{\xi}]\subseteq \text{Ker}(\xi)$ for all $\xi\in\Delta.$
\end{lemma}
\begin{proof} Let $\xi\in \Delta$ be fixed and set $S_{\xi}:=\mathfrak{h}\oplus\mathfrak{g}_{\xi}$ which is a soluble subalgebra of $\mathfrak{g}$. Suppose that there is  $h\in[\mathfrak{g}_\xi,\mathfrak{g}_\xi]$ such that  $\xi(h)\neq0$. As $\mathfrak{g}_\xi\subseteq S_{\xi}$, we have $[\mathfrak{g}_\xi,\mathfrak{g}_\xi]\subseteq[{S}_{\xi},{S}_\xi]={{S}_\xi}^{(1)}$, hence $h\in {S_{\xi}}^{(1)}$.  Thus  $[h,S_\xi]\subseteq[{S_{\xi}}^{(1)},S_{\xi}]\subseteq{S_\xi}^{(1)}$, since ${S_{\xi}}^{(1)}$ is an ideal of $S_{\xi}$. Now, take $v\in \mathfrak{g}_{\xi}$. Given that  $\xi(h)\neq 0$, we have $v=\frac{1}{\xi(h)}[v,h]\in[h,S_{\xi}]\subseteq{S_{\xi}}^{(1)}.$
Therefore $\mathfrak{g}_\xi\subseteq {S_\xi}^{(1)}$ and so \[h\in[\mathfrak{g}_\xi,\mathfrak{g}_\xi]\subseteq [{S_\xi}^{(1)},{S_\xi}^{(1)}]={S_\xi}^{(2)}.\]

Next,  suppose that $h\in {S_{\xi}}^{(m)}$ for some $m\in\mathbb{N}$. As ${S_{\xi}}^{(m)}$ is an ideal of $S_{\xi}$, we have  $[h,S_{\xi}]\subseteq{S_{\xi}}^{(m)}$ and  hence  $\mathfrak{g}_{\xi}\subseteq {S_\xi}^{(m)}$ (since $\xi(t)\neq0$). Thus  \[h\in[\mathfrak{g}_\xi,\mathfrak{g}_\xi]\subseteq [{S_\xi}^{(m)},{S_\xi}^{(m)}]={S_\xi}^{(m+1)}\] 
and so $h\in{S_\xi}^{(m+1)}$. Then by induction, $h\in{S_\xi}^{(m)} $ for all $m$, and since $S_\xi$ is solvable, we have $h=0$. This fact implies that  \,$\xi(h)=0$, which is a contradiction. Therefore $\xi([\mathfrak{g}_\xi,\mathfrak{g}_\xi])=0$. 
\end{proof}

\begin{proposition}\label{payarin}
There are no simple Lie $2$-algebras of finite dimensional of toral rank $3$, with Cartan decomposition   $\mathfrak{g}^{\Delta_{i}}=\mathfrak{h}\oplus(\underset{\xi\in\Delta_{i}}\oplus \mathfrak{g}_\xi)$, for $1\leq i\leq15$.   In particular, 
there are no simple Lie $2$-algebras $\mathfrak{g}$ with $7\leq \text{dim}_{K}(\mathfrak{g})\leq 9$ and  $MT(\mathfrak{g})=3$.
\end{proposition}
\begin{proof}
 If $\text{dim}_{K}(\mathfrak{g}_\xi)=1$ for all $\xi\in\Delta_i$, with  $1\leq i\leq15$,  we have that $$0\neq \mathfrak{I}:=\mathfrak{n} \oplus(\underset{\xi\in\Delta_i}\oplus \mathfrak{g}_\xi)$$  is a non-trivial ideal  of $\mathfrak{g}$, since $[\mathfrak{n},\mathfrak{g}_\xi]=0$ (by Lemma \ref{edfrr}) and $[\mathfrak{g}_\xi,\mathfrak{g}_\xi]=0$. This fact  is a contradiction,  because $\mathfrak{g}$ is a simple Lie algebra. So,  we can  assume that $\text{dim}_K(\mathfrak{g}_\xi)\geq 2$ for all $\xi\in\Delta_i$.
 
 We will suppose, by contradiction, that there exist simple Lie 2-algebras of toral rank 3, with Cartan decomposition given in \eqref{jeo1} 
and we will get a contradiction. 
 
 \medskip
 
 \noindent \textbf{Case 1}:  $\Delta_1=\{\alpha,\beta,\gamma\}$.
 
 In this case  we have \[\mathfrak{g}=\mathfrak{t}\oplus\mathfrak{n}\oplus\mathfrak{g}_{\alpha} \oplus\mathfrak{g}_{\beta}\oplus \mathfrak{g}_{\gamma}.\]
 First we show that $[\mathfrak{g}_\xi,\mathfrak{g}_\xi]\subseteq \mathfrak{n}$, for all $\xi\in\Delta_1.$ Indeed, let $e_i^{\xi}$ and $e_j^{\xi}$ be elements in the basis of $\mathfrak{g}_\xi$. Thus ${x}^{\xi}:=[e_i^{\xi},e_j^{\xi}]\in[\mathfrak{g}_\xi,\mathfrak{g}_\xi]\subseteq \mathfrak{h}=\mathfrak{n}\oplus \mathfrak{t}$ and therefore there are $t_{ij}^{\xi}\in \mathfrak{t}$ and $n_{ij}^{\xi}\in \mathfrak{n}$ such that   $x^{\xi}=t_{ij}^{\xi}+n_{ij}^{\xi}=\alpha(t_{ij}^{\xi})t_1+\beta(t_{ij}^{\xi})t_2+\gamma(t_{ij}^{\xi})t_3+n_{ij}^\xi.$ By  Lemma \ref{patrondepatrones} we have $0=\xi(x^{\xi})=\xi(t_{ij}^{\xi}+n_{ij}^{\xi})$ for all $\xi\in\Delta_{1}$. Hence  $0=\xi(x^{\xi})=\xi(t_{ij}^{\xi}+n_{ij}^{\xi})=\xi(t_{ij}^\xi)$. Then $\alpha(t_{ij}^{\alpha})=\beta(t_{ij}^{\beta})=\gamma(t_{ij}^{\gamma})=0$ and therefore $x^{\xi}=n_{ij}^{\xi}\in \mathfrak{n}$ for all $\xi\in\Delta_{1}$. The simplicity of $\mathfrak{g}$, implies that  $\mathfrak{h}=\sum_{\xi\in\Delta_1}[\mathfrak{g}_\xi,\mathfrak{g}_\xi]$,  so $\mathfrak{h}\subseteq \mathfrak{n}$. Therefore $\mathfrak{h}=\mathfrak{n}$ and $\mathfrak{t}=0$, which is a contradiction.
 
 \medskip
 
 \noindent \textbf{Case 2.a}:  $\Delta_{2}=\{\alpha,\beta,\gamma,\alpha+\beta\}$. 
 
 We have  \[\mathfrak{g}=\mathfrak{t}\oplus \mathfrak{n}\oplus \mathfrak{g}_{\alpha} \oplus \mathfrak{g}_{\beta}\oplus \mathfrak{g}_{\gamma} \oplus \mathfrak{g}_{\alpha+\beta}.\]
 Let  $e_i^{\xi}$ and $e_j^{\xi}$ be elements of the base of $\mathfrak{g}_\xi$. Then ${x}^{\xi}:=[e_i^{\xi},e_j^{\xi}]\in \mathfrak{h}$ and there are $t_{ij}^{\xi}\in \mathfrak{t}$ and $n_{ij}^{\xi}\in \mathfrak{n}$ such that $x^{\xi}=t_{ij}^{\xi}+n_{ij}^{\xi}$. So $$0=\xi(x^{\xi})=\xi(t_{ij}^{\xi}+n_{ij}^{\xi})=\xi(t_{ij}^{\xi})\quad\text{
for all }\xi\in\Delta_2,$$ that is \[ \alpha(t_{ij}^{\alpha})=\beta(t_{ij}^{\beta})=\gamma(t_{ij}^{\gamma})=(\alpha+\beta)(t_{ij}^{\alpha+\beta})= 0.\]

On the other hand, by the Jacobi identity and since $\alpha+\gamma\notin\Delta_{2}$, we have that   \begin{equation*}0=[e_{k}^{\gamma},[e_{i}^{\alpha},e_{j}^{\alpha}]]= [e_{k}^{\gamma},x^{\alpha}]
=[e_{k}^{\gamma},t_{ij}^{\alpha}]+[e_{k}^{\gamma},n_{ij}^{\alpha}]=\gamma(t_{ij}^{\alpha})e_{k}^{\gamma}+\text{ad}(n_{ij}^{\alpha})(e_{k}^{\gamma}),
\end{equation*}
 for all $e_{k}^{\gamma}\in \mathfrak{g}_{\gamma}$.
So $\text{ad}(n_{ij}^{\alpha})(e_{k}^{\gamma})=\gamma(t_{ij}^{\alpha})e_{k}^{\gamma}$. Since $\text{ad}(n_{ij}^{\alpha})$ is nilpotent, there is $m\in \mathbb{N} $ such that $\text{ad}(n_{ij}^{\alpha})^{m}=0$ and then $\gamma(t_{ij}^{\alpha})=0$. Hence, \[x^{\alpha}=t_{ij}^{\alpha}+n_{ij}^{\alpha}= \alpha(t_{ij}^{\alpha})t_1+\beta(t_{ij}^{\alpha})t_2+\gamma(t_{ij}^{\alpha})t_3+n_{ij}^{\alpha}=\beta(t_{ij}^{\alpha})t_2+n_{ij}^{\alpha}.\]
Therefore $[\mathfrak{g}_\alpha,\mathfrak{g}_\alpha]\subseteq \left\langle t_2\right\rangle\oplus \mathfrak{n}$.
Analogously,    by the Jacobi identity and since $\alpha+\gamma,\beta+\gamma,\alpha+\beta+\gamma\notin \Delta_2$, we can prove that  $\alpha(t_{ij}^{\gamma})=\beta(t_{ij}^{\gamma})=\gamma(t_{ij}^{\beta})=\gamma(t_{ij}^{\alpha+\beta})=0$.
Therefore
\begin{align*}
x^{\beta}&=\alpha(t_{ij}^{\beta})t_1 + n_{ij}^{\beta} \\
x^{\gamma}&=n_{ij}^{\gamma}\\ 
x^{\alpha+\beta}&=\alpha(t_{ij}^{\alpha+\beta})t_1+\beta(t_{ij}^{\alpha+\beta})t_2 + n_{ij}^{\alpha+\beta},\quad  \alpha(t_{ij}^{\alpha+\beta})=\beta(t_{ij}^{\alpha+\beta}).
\end{align*}

 
Thus \[[\mathfrak{g}_{\beta},\mathfrak{g}_{\beta}]\subseteq\left\langle t_1\right\rangle\oplus\mathfrak{n}, \quad  [\mathfrak{g}_\gamma,\mathfrak{g}_\gamma]\subseteq \mathfrak{n}\quad\text{ and }\quad[\mathfrak{g}_{\alpha+\beta+\gamma},\mathfrak{g}_{\alpha+\beta+\gamma}]\subseteq\left\langle t_1+t_2\right\rangle\oplus \mathfrak{n}.\]
As $\mathfrak{g}$ is simple, we have $\mathfrak{h}=\sum_{\xi\in\Delta_2}[\mathfrak{g}_\xi,\mathfrak{g}_\xi]\subseteq\left\langle t_1,t_2\right\rangle\oplus \mathfrak{n}$ and then $\text{dim}_{K}(\mathfrak{t})\leq2$, which is a contradiction. 

\medskip

 \noindent \textbf{Case 2.b}:  $\Delta_{3}=\{\alpha,\beta,\gamma,\alpha+\gamma\}$.  
 
 We have \[\mathfrak{g}=\mathfrak{t}\oplus \mathfrak{n}\oplus \mathfrak{g}_{\alpha} \oplus \mathfrak{g}_{\beta}\oplus \mathfrak{g}_{\gamma} \oplus \mathfrak{g}_{\alpha+\gamma}.\] By the same argument of the above case and since $\alpha+\beta,\beta+\gamma$, and $\alpha+\beta+\gamma$ do not belong to $\Delta_3$, we find that:
\begin{align*}
 x^{\alpha}&=\gamma(t_{ij}^{\alpha})t_3 + n_{ij}^{\alpha}.\\
 x^{\beta}&=n_{ij}^{\beta}.\\
 x^{\gamma}&=\alpha(t_{ij}^{\gamma})t_1 + n_{ij}^{\gamma}.\\
 x^{\alpha+\gamma}&=\alpha(t_{ij}^{\alpha+\gamma})t_1+\gamma(t_{ij}^{\alpha+\gamma})t_3 + n_{ij}^{\alpha+\gamma},\quad \alpha(t_{ij}^{\alpha+\gamma})=\gamma(t_{ij}^{\alpha+\gamma}).
\end{align*}
Therefore
\begin{align*}
[\mathfrak{g}_\alpha,\mathfrak{g}_\alpha]&\subseteq\left\langle t_3\right\rangle\oplus \mathfrak{n}  \\
[\mathfrak{g}_\beta,\mathfrak{g}_\beta]&\subseteq \mathfrak{n}\\
[\mathfrak{g}_\gamma,\mathfrak{g}_\gamma]&\subseteq\left\langle t_1\right\rangle\oplus \mathfrak{n}\\
[\mathfrak{g}_{\alpha+\gamma},\mathfrak{g}_{\alpha+\gamma}]&\subseteq\left\langle t_1+t_3\right\rangle\oplus \mathfrak{n}
\end{align*}


Hence $\mathfrak{h}=\sum_{\xi\in\Delta_3}[\mathfrak{g}_\xi,\mathfrak{g}_\xi]\subseteq\left\langle t_1,t_3\right\rangle\oplus \mathfrak{n}$ and then  $\text{dim}_{K}(\mathfrak{t})\leq2$, which is a contradiction.

\medskip
\noindent \textbf{Case 2.c}: $\Delta_{4}=\{\alpha,\beta,\gamma,\beta+\gamma\}$.

Thus \[\mathfrak{g}=\mathfrak{t}\oplus \mathfrak{n}\oplus \mathfrak{g}_{\alpha} \oplus \mathfrak{g}_{\beta}\oplus \mathfrak{g}_{\gamma} \oplus \mathfrak{g}_{\beta+\gamma}.\] Since $\alpha+\beta, \alpha+\gamma, \alpha+\beta+\gamma \notin \Delta_4$ and using the same argument of the above cases we have 
\begin{align*}
 x^{\alpha}&=n_{ij}^{\alpha}\\
 x^{\beta}&=\beta(t_{ij}^{\beta})t_3 + n_{ij}^{\beta}\\
 x^{\gamma}&=\beta(t_{ij}^{\gamma})t_2 + n_{ij}^{\gamma}.\\
 x^{\beta+\gamma}&=\beta(t_{ij}^{\beta+\gamma})t_2+\gamma(t_{ij}^{\beta\gamma})t_3 + n_{ij}^{\beta+\gamma},\quad \beta(t_{ij}^{\beta+\gamma})=\gamma(t_{ij}^{\beta+\gamma}).
\end{align*}
Therefore 
\begin{align*}
 [\mathfrak{g}_\alpha,\mathfrak{g}_\alpha]&\subseteq \mathfrak{n}\\ 
 [\mathfrak{g}_\beta,\mathfrak{g}_\beta]&\subseteq\left\langle t_3\right\rangle\oplus \mathfrak{n}\\
 [\mathfrak{g}_\gamma,\mathfrak{g}_\gamma]&\subseteq\left\langle t _2\right\rangle\oplus \mathfrak{n},\\
[\mathfrak{g}_{\beta+\gamma},\mathfrak{g}_{\beta+\gamma}]&\subseteq\left\langle t_2+t_3\right\rangle\oplus \mathfrak{n}.
\end{align*}

Thus $\mathfrak{h}=\sum_{\xi\in\Delta_4}[\mathfrak{g}_\xi,\mathfrak{g}_\xi]\subseteq\left\langle t_2,t_3\right\rangle \oplus \mathfrak{n}$ and  hence $\text{dim}_{K}(\mathfrak{t})\leq2$. A contradiction.

\medskip

\noindent \textbf{Case 2.b}: $\Delta_{5}=\{\alpha,\beta,\gamma,\alpha+\beta+\gamma\}$. 

Thus \[\mathfrak{g}=\mathfrak{t}\oplus \mathfrak{n}\oplus \mathfrak{g}_{\alpha} \oplus \mathfrak{g}_{\beta}\oplus \mathfrak{g}_{\gamma} \oplus \mathfrak{g}_{\alpha+\beta+\gamma}.\] We have  $\alpha+\beta, \alpha+\gamma, \beta+\gamma \notin \Delta_5$. This fact implies that  
$[\mathfrak{g}_\alpha,\mathfrak{g}_\alpha]$, $[\mathfrak{g}_\beta,\mathfrak{g}_\beta]$, $[\mathfrak{g}_\gamma,\mathfrak{g}_\gamma]$ and  $[\mathfrak{g}_{\alpha+\beta+\gamma}, \mathfrak{g}_{\alpha+\beta+\gamma}]$ are subsets of $  \mathfrak{n}.$  Therefore $\mathfrak{h}\subseteq \mathfrak{n}$ and thus $\mathfrak{t}=0$. A contradiction.

\medskip

\noindent \textbf{Case 3.a}:  $\Delta_{6}=\{\alpha,\beta,\gamma,\alpha+\beta,\alpha+\gamma\}$.

We have \[\mathfrak{g}=\mathfrak{t}\oplus \mathfrak{n}\oplus \mathfrak{g}_{\alpha} \oplus\mathfrak{g}_{\beta}\oplus \mathfrak{g}_{\gamma} \oplus \mathfrak{g}_{\alpha+\beta}\oplus \mathfrak{g}_{\alpha+\gamma}.\] 
We can prove that 
\begin{align*}
 x^{\alpha}&=\beta(t_{ij}^{\alpha})t_2+ \gamma(t_{ij}^{\alpha})t_3+n_{ij}^{\alpha}. \\
 x^{\beta}&=n_{ij}^{\beta}.\\
 x^\gamma &=n_{ij}^\gamma.\\
 x^{\alpha+\beta}&=n_{ij}^{\alpha+\beta}.\\
 x^{\alpha+\gamma}&=n_{ij}^{\alpha+\gamma}.
  \end{align*}

Therefore,
$[\mathfrak{g}_\xi,\mathfrak{g}_\xi]\subseteq \mathfrak{n}$ for all $\xi\in\Delta_6\backslash\{\alpha\}$ and 
$[\mathfrak{g}_\alpha,\mathfrak{g}_\alpha]\subseteq\left\langle t_2+t_3\right\rangle\oplus \mathfrak{n}.$ So $$\mathfrak{g}=\sum_{\xi\in\Delta_6}[\mathfrak{g}_\xi,\mathfrak{g}_\xi]\subseteq\left\langle t_2,t_3\right\rangle\oplus \mathfrak{n}.$$  Hence $\text{dim}_{K}(\mathfrak{t})\leq2$. A contradiction.

\medskip

\noindent \textbf{Case 3.b}:   $\Delta_{7}=\{\alpha,\beta,\gamma,\alpha+\beta,\beta+\gamma\}$. 

We have \[\mathfrak{g}=\mathfrak{t}\oplus \mathfrak{n}\oplus \mathfrak{g}_{\alpha} \oplus \mathfrak{g}_{\beta}\oplus \mathfrak{g}_{\gamma} \oplus \mathfrak{g}_{\alpha+\beta}\oplus \mathfrak{g}_{\beta+\gamma}.\] 
We  find that: 
\begin{align*}
 [\mathfrak{g}_\alpha,\mathfrak{g}_\alpha]&\subseteq\mathfrak{n}.\\
 [\mathfrak{g}_\beta,\mathfrak{g}_\beta]&\subseteq\left\langle t_1+t_3\right\rangle \oplus \mathfrak{n}.\\
 [\mathfrak{g}_\gamma,\mathfrak{g}_\gamma]&\subseteq \mathfrak{n}.\\
 [\mathfrak{g}_{\alpha+\beta},\mathfrak{g}_{\alpha+\beta}]&\subseteq \left\langle t_1+t_2\right\rangle\oplus\mathfrak{n}.\\ 
 [\mathfrak{g}_{\beta+\gamma},\mathfrak{g}_{\beta+\gamma}]&\subseteq\left\langle t_2+t_3\right\rangle\oplus \mathfrak{n}.
\end{align*}

Set  $$\mathfrak{t}_0:=\left\langle t_1+t_3\right\rangle \oplus \left\langle t_1+t_2\right\rangle \oplus \left\langle t_2+t_3\right\rangle.$$  Thus $\mathfrak{h}\subseteq\sum_{\xi\in\Delta_{7}}[\mathfrak{g}_\xi,\mathfrak{g}_\xi]\subseteq \mathfrak{t}_{0}\oplus \mathfrak{n}$. Let's see that this fact implies that
$t_1\notin[\mathfrak{g},\mathfrak{g}].$ Indeed, if $t_1\in [\mathfrak{g},\mathfrak{g}]\subseteq\sum_{\xi\in\Delta_{7}}[\mathfrak{g}_\xi,\mathfrak{g}_\xi]\subseteq \mathfrak{t}_{0}\oplus \mathfrak{n}$, then there are $\delta_1$, $\delta_2$,  $\delta_3$ in $K$ and $n\in\mathfrak{n}$ such that $$t_1=\delta_1(t_1+t_3)+\delta_2(t_1+t_2)+\delta_3(t_2+t_3)+n,$$ that is,  $$(1+\delta_1+\delta_2)t_1+(\delta_2+\delta_3)t_2+(\delta_1+\delta_3)t_3=n.$$ Since $\mathfrak{t}\cap \mathfrak{n}=0$, we have $(1+\delta_1+\delta_2)t_1+(\delta_2+\delta_3)t_2+(\delta_1+\delta_3)t_3=0$. As $t_1$, $t_2$, $t_3$ are linearly independent,   we obtain the following system of equations:  
\[ \left\{
\begin{array}{lll}
1+\delta_1+\delta_2 & =0
\\
\delta_2+\delta_3 &=0
 \\
\delta_1+\delta_3 & =0 .
 \end{array}
\right.
\]
Solving the system, we have that $0\delta_1+0\delta_2=1$, which is absurd. Therefore $t_1\notin[\mathfrak{g},\mathfrak{g}]$. However, this fact implies  that $\mathfrak{t}\not\subseteq [\mathfrak{g},\mathfrak{g}],$ which is absurd, since $\mathfrak{t}\subseteq[\mathfrak{g},\mathfrak{g}].$

\medskip

\noindent \textbf{Case 3.c}:      $\Delta_{8}=\{\alpha,\beta,\gamma,\alpha+\beta,\alpha+\beta+\gamma\}$. 

We have  \[\mathfrak{g}=\mathfrak{t}\oplus \mathfrak{n}\oplus \mathfrak{g}_{\alpha} \oplus \mathfrak{g}_{\beta}\oplus \mathfrak{g}_{\gamma} \oplus \mathfrak{g}_{\alpha+\beta}\oplus \mathfrak{g}_{\alpha+\beta+\gamma}.\]
So,
\begin{align*}
 [\mathfrak{g}_\xi,\mathfrak{g}_\xi]&\subseteq \mathfrak{n} \quad \forall \xi\in\Delta_8\backslash\{\alpha+\beta\}.\\
 [\mathfrak{g}_{\alpha+\beta},\mathfrak{g}_{\alpha+\beta}]&\subseteq\left\langle t_1+t_2,t_3\right\rangle\oplus \mathfrak{n}.
\end{align*}
Then $\mathfrak{h}\subseteq\left\langle t_1+t_2,t_3\right\rangle\oplus \mathfrak{n}$ and  so $\text{dim}_K(\mathfrak{h})\leq \text{dim}_K(\mathfrak{n})+2$. Hence  $\text{dim}_{K}(\mathfrak{t})\leq2$ (contradiction).

\medskip

\noindent \textbf{Case 3.d}:     $\Delta_{9}=\{\alpha,\beta,\gamma,\alpha+\gamma,\beta+\gamma\}$. 

Then \[{\mathfrak{g}=\mathfrak{t}\oplus \mathfrak{n}\oplus \mathfrak{g}_{\alpha} \oplus \mathfrak{g}_{\beta}\oplus \mathfrak{g}_{\gamma} \oplus \mathfrak{g}_{\alpha+\gamma}\oplus \mathfrak{g}_{\beta+\gamma}}.\] We have
\begin{align*}
[\mathfrak{g}_\alpha,\mathfrak{g}_\alpha]&\subseteq \mathfrak{n}.\\ 
 [\mathfrak{g}_\beta,\mathfrak{g}_\beta]&\subseteq \mathfrak{n}.\\
 [\mathfrak{g}_\gamma,\mathfrak{g}_\gamma]&\subseteq\left\langle t_1,t_2\right\rangle\oplus \mathfrak{n}.\\
 [\mathfrak{g}_{\alpha+\gamma},\mathfrak{g}_{\alpha+\gamma}]&\subseteq \mathfrak{n}.\\
 [\mathfrak{g}_{\beta+\gamma},\mathfrak{g}_{\beta+\gamma}]&\subseteq \mathfrak{n}.
\end{align*}
Therefore $\mathfrak{h}=\sum_{\xi\in\Delta_9}[\mathfrak{g}_\xi,\mathfrak{g}_\xi]\subseteq\left\langle t_1,t_2\right\rangle\oplus \mathfrak{n}$ and then $\text{dim}_k(\mathfrak{t})\leq2$. A contradiction.

\medskip

\noindent \textbf{Case 3.e}:     $\Delta_{10}=\{\alpha,\beta,\gamma,\alpha+\gamma,\alpha+\beta+\gamma\}$.

Then \[\mathfrak{g}=\mathfrak{t}\oplus \mathfrak{n}\oplus \mathfrak{g}_{\alpha} \oplus \mathfrak{g}_{\beta}\oplus \mathfrak{g}_{\gamma} \oplus \mathfrak{g}_{\alpha+\gamma}\oplus \mathfrak{g}_{\alpha+\beta+\gamma}.\]
We have, 
\begin{align*}
[\mathfrak{g}_\xi,\mathfrak{g}_\xi]&\subseteq \mathfrak{n} \quad\forall\xi\in\Delta_{10}\backslash\{\alpha+\gamma\} \\
[\mathfrak{g}_{\alpha+\gamma},\mathfrak{g}_{\alpha+\gamma}]&\subseteq\left\langle t_2,t_1+t_3 \right\rangle\oplus\mathfrak{n} ,
\end{align*}
and then  $\text{dim}_k(T)\leq2$, which is absurd.

\medskip

\noindent \textbf{Case 3.f}:     $\Delta_{11}=\{\alpha,\beta,\gamma,\beta+\gamma,\alpha+\beta+\gamma\}$. 

Then \[\mathfrak{g}=\mathfrak{t}\oplus \mathfrak{n}\oplus \mathfrak{g}_{\alpha} \oplus \mathfrak{g}_{\beta}\oplus \mathfrak{g}_{\gamma} \oplus \mathfrak{g}_{\beta+\gamma}\oplus \mathfrak{g}_{\alpha+\beta+\gamma}.\] Therefore
\begin{align*}
 [\mathfrak{g}_\xi,\mathfrak{g}_\xi]&\subseteq \mathfrak{n}\quad 
 \forall\xi\in\Delta_{11}\backslash\{\beta+\gamma\}.\\
 [\mathfrak{g}_{\beta+\gamma},\mathfrak{g}_{\beta+\gamma}]&\subseteq\left\langle t_2,t_1+t_3\right\rangle\oplus \mathfrak{n}.
\end{align*}

Hence  $\mathfrak{h}\subseteq\left\langle t_2,t_1+t_3\right\rangle\oplus \mathfrak{n}$ and  so $\text{dim}(\mathfrak{t})\leq2$. A contradiction.

\medskip

\noindent \textbf{Case 4.a}:    $\Delta_{12}=\{\alpha,\beta,\gamma,\alpha+\beta,\alpha+\gamma,\beta+\gamma\}$. 

Then \[\mathfrak{g}=\mathfrak{t}\oplus \mathfrak{n}\oplus \mathfrak{g}_{\alpha} \oplus \mathfrak{g}_{\beta}\oplus \mathfrak{g}_{\gamma} \oplus \mathfrak{g}_{\alpha+\beta}\oplus \mathfrak{g}_{\alpha+\gamma}\oplus \mathfrak{g}_{\beta+\gamma}.\]
By  similar arguments  presented in the above cases, we have that
\begin{align*}
[\mathfrak{g}_\alpha,\mathfrak{g}_\alpha],\quad[\mathfrak{g}_{\beta+\gamma},\mathfrak{g}_{\beta+\gamma}]&\subseteq\left\langle t_2+t_3\right\rangle\oplus \mathfrak{n}.\\
[\mathfrak{g}_\beta,\mathfrak{g}_\beta],\quad[\mathfrak{g}_{\alpha+\gamma},\mathfrak{g}_{\alpha+\gamma}]&\subseteq\left\langle t_1+t_3\right\rangle \oplus \mathfrak{n}.\\
[\mathfrak{g}_\gamma,\mathfrak{g}_\gamma],\quad[\mathfrak{g}_{\alpha+\beta},\mathfrak{g}_{\alpha+\beta}]&\subseteq\left\langle t_1+t_2\right\rangle\oplus \mathfrak{n}.
\end{align*}
As in  Case 3.b, this fact implies that $t_1\notin[\mathfrak{g},\mathfrak{g}]$, because  if   \[t_1\in[\mathfrak{g},\mathfrak{g}]\subseteq\sum_{\xi\in\Delta_{12}}[\mathfrak{g}_\xi,\mathfrak{g}_\xi]\subseteq \mathfrak{t}_{0}\oplus \mathfrak{n},\quad \text{where}\quad \mathfrak{t}_{0}:=\left\langle t_2+t_3 \right\rangle\oplus \left\langle t_1+t_2\right\rangle\oplus \left\langle t_1+t_3\right\rangle,\]  then, for some $\delta_i\in K$ and $n\in \mathfrak{n}$,  \[t_1=\delta_1(t_2+t_3)+\delta_2(t_1+t_2)+\delta_3(t_1+t_3)+ n.\]  Therefore $(1+\delta_2+\delta_3)t_1+(\delta_1+\delta_2)t_2+(\delta_1+\delta_3)t_3=n$. Since $\mathfrak{t}\cap \mathfrak{n}=0$, we have $(1+\delta_2+\delta_3)t_1+(\delta_1+\delta_2)t_2+(\delta_1+\delta_3)t_3=n=0$. 
Given that $t_1,t_2,t_3$ are linearly independent,  we obtain  the following system of equations
\[  \left\{
\begin{array}{lll}
1+\delta_2+\delta_3 & =0
\\
\delta_1+\delta_2 &=0
 \\
\delta_1+\delta_3 & =0 
 \end{array}
\right.
\]
which has no solution in $K$, and this is a contradiction. Therefore $t_1\notin[\mathfrak{g},\mathfrak{g}]$, that is, $\mathfrak{t}\not\subseteq [\mathfrak{g},\mathfrak{g}],$ which is absurd, since $\mathfrak{t}\subseteq[\mathfrak{g},\mathfrak{g}].$ 

\medskip

\noindent \textbf{Case 4.b}:      $\Delta_{13}=\{\alpha,\beta,\gamma,\alpha+\beta,\alpha+\gamma,\alpha+\beta+\gamma\}$. 

Then \[\mathfrak{g}=\mathfrak{t}\oplus\mathfrak{n} \oplus \mathfrak{g}_{\alpha} \oplus \mathfrak{g}_{\beta}\oplus \mathfrak{g}_{\gamma} \oplus \mathfrak{g}_{\alpha+\beta}\oplus \mathfrak{g}_{\alpha+\gamma}\oplus \mathfrak{g}_{\alpha+\beta+\gamma}.\] We have
\begin{align*}
  [\mathfrak{g}_\alpha,\mathfrak{g}_\alpha],[\mathfrak{g}_{\alpha+\beta+\gamma},\mathfrak{g}_{\alpha+\beta+\gamma}]&\subseteq\left\langle t_2+t_3\right\rangle\oplus \mathfrak{n}.\\
  [\mathfrak{g}_\beta,\mathfrak{g}_\beta],[\mathfrak{g}_{\gamma},\mathfrak{g}_{\gamma}]&\subseteq\left\langle t_1\right\rangle\oplus \mathfrak{n}\\
  [\mathfrak{g}_{\alpha+\beta},\mathfrak{g}_{\alpha+\beta}],[\mathfrak{g}_{\alpha+\gamma},\mathfrak{g}_{\alpha+\gamma}]&\subseteq\left\langle t_1+t_2+t_3\right\rangle\oplus \mathfrak{n}.
\end{align*}
By  the same argument given in  Case 3.b, we have that $t_{2}\notin[\mathfrak{g},\mathfrak{g}]$, which is a contradiction,  since $\mathfrak{t}\subseteq[\mathfrak{g},\mathfrak{g}].$

\medskip

\noindent \textbf{Case 4.c}:     $\Delta_{14}=\{\alpha,\beta,\gamma,\alpha+\beta,\beta+\gamma,\alpha+\beta+\gamma\}$.

We have \[\mathfrak{g}=\mathfrak{t}\oplus \mathfrak{n}\oplus \mathfrak{g}_{\alpha} \oplus \mathfrak{g}_{\beta}\oplus \mathfrak{g}_{\gamma} \oplus \mathfrak{g}_{\alpha+\beta}\oplus \mathfrak{g}_{\beta+\gamma}\oplus \mathfrak{g}_{\alpha+\beta+\gamma}.\]
Therefore,
\begin{align*}
[\mathfrak{g}_\alpha,\mathfrak{g}_\alpha]&\subseteq\left\langle t_2\right\rangle\oplus \mathfrak{n} \\ 
[\mathfrak{g}_\beta,\mathfrak{g}_\beta]&\subseteq\left\langle t_1+t_3\right\rangle\oplus \mathfrak{n}\\
[\mathfrak{g}_\gamma,\mathfrak{g}_\gamma]&\subseteq\left\langle t_2\right\rangle\oplus \mathfrak{n}\\
[\mathfrak{g}_{\alpha+\beta},\mathfrak{g}_{\alpha+\beta}]&\subseteq\left\langle t_1+t_2+t_3\right\rangle\oplus \mathfrak{n}\\
[\mathfrak{g}_{\beta+\gamma},\mathfrak{g}_{\beta+\gamma}]&\subseteq\left\langle t_1+t_2+t_3 \right\rangle\oplus\mathfrak{n}\\
[\mathfrak{g}_{\alpha+\beta+\gamma},\mathfrak{g}_{\alpha+\beta+\gamma}]&\subseteq\left\langle t_1+t_3\right\rangle\oplus \mathfrak{n}
\end{align*}

We prove that $t_1\notin[\mathfrak{g},\mathfrak{g}]$. Otherwise, 
 \[t_1\in[\mathfrak{g},\mathfrak{g}]\subseteq\sum_{\xi\in\Delta_{14}}[\mathfrak{g}_\xi,\mathfrak{g}_\xi]\subseteq \mathfrak{t}_{0}\oplus \mathfrak{n},\quad\text{where}\quad \mathfrak{t}_{0}:=\left\langle t_2\right\rangle\oplus \left\langle t_1+t_3\right\rangle\oplus\left\langle t_1+t_2+t_3\right\rangle.\]
 Then $t_1=\delta_1(t_2)+\delta_2(t_1+t_3)+\delta_3(t_1+t_2+t_3)+ n$ for some $\delta_i\in K$ and $n\in \mathfrak{n}$. Thus $(1+\delta_2+\delta_3)t_1+(\delta_1+\delta_2)t_2+(\delta_1+\delta_3)t_3=n$. As $\mathfrak{t}\cap\mathfrak{n}=0$, then $(1+\delta_2+\delta_3)t_1+(\delta_1+\delta_2)t_2+(\delta_1+\delta_3)t_3=n=0.$ 
Since $t_1$, $t_2$, $t_3$ are linearly independent,  we have the following system 
\[  \left\{
\begin{array}{lll}
1+\delta_2+\delta_3 & =0
\\
\delta_1+\delta_2 &=0
 \\
\delta_1+\delta_3 & =0 
 \end{array}
\right.
\]
which has no solution,   a contradiction. Hence, $t_1\notin[\mathfrak{g},\mathfrak{g}]$. This fact is absurd because   $\mathfrak{t}\subseteq[\mathfrak{g},\mathfrak{g}].$

\medskip

\noindent \textbf{Case 4.d}:   $\Delta_{15}=\{\alpha,\beta,\gamma,\alpha+\gamma,\beta+\gamma,\alpha+\beta+\gamma\}$, then \[\mathfrak{g} =\mathfrak{t}\oplus\mathfrak{n} \oplus \mathfrak{g}_{\alpha} \oplus\mathfrak{g}_{\beta}\oplus \mathfrak{g}_{\gamma} \oplus \mathfrak{g}_{\alpha+\gamma}\oplus \mathfrak{g}_{\beta+\gamma}\oplus \mathfrak{g}_{\alpha+\beta+\gamma}\] We have
\begin{align*}
 [\mathfrak{g}_\alpha,\mathfrak{g}_\alpha]&\subseteq\left\langle t_3\right\rangle\oplus \mathfrak{n}.\\
 [\mathfrak{g}_\beta,\mathfrak{g}_\beta]&\subseteq\left\langle t_3\right\rangle \oplus \mathfrak{n}.\\
 [\mathfrak{g}_\gamma,\mathfrak{g}_\gamma]&\subseteq\left\langle t_1+t_2\right\rangle \oplus \mathfrak{n}.\\
 [\mathfrak{g}_{\alpha+\gamma},\mathfrak{g}_{\alpha+\gamma}]&\subseteq\left\langle t_1+t_2+t_3\right\rangle\oplus \mathfrak{n}.\\
 [L_{\beta+\gamma},L_{\beta+\gamma}]&\subseteq\left\langle t_1+t_2+t_3\right\rangle\oplus\mathfrak{n}.\\
 [\mathfrak{g}_{\alpha+\beta+\gamma},\mathfrak{g}_{\alpha+\beta+\gamma}]&\subseteq\left\langle t_1+t_2\right\rangle\oplus \mathfrak{n}.
\end{align*}
Set $\mathfrak{t}_{0}=\left\langle t_3\right\rangle\oplus \left\langle t_1+t_2\right\rangle\oplus\left\langle t_1+t_2+t_3\right\rangle$. Thus $[\mathfrak{g},\mathfrak{g}]\subseteq\sum_{\xi\in\Delta_{15}}[\mathfrak{g}_\xi,\mathfrak{g}_\xi]\subseteq \mathfrak{t}_{0}\oplus\mathfrak{n}$. If $t_1\in [\mathfrak{g},\mathfrak{g}]$ then there are  $\delta_1,\delta_2, \delta_3$ in  $K$ such that $t_1=\delta_1t_3+\delta_2(t_1+t_2)+\delta_3(t_1+t_2+t_3)+n$. Therefore we have the following system of equations
\[  \left\{
\begin{array}{lll}
1+\delta_2+\delta_3 & =0
\\
\delta_2+\delta_3 &=0
 \\
\delta_1+\delta_3 & =0 .
 \end{array}
\right.
\]
Solving the system, we have $0\delta_2+0\delta_3=1$, which is absurd. Therefore $t_1\notin[\mathfrak{g},\mathfrak{g}]$, which  is a contradiction because $\mathfrak{t}\subseteq[\mathfrak{g},\mathfrak{g}]$. \end{proof}

\section{Analysis of $\mathfrak{g}=\mathfrak{h}\oplus( \oplus _{\xi\in\Delta}\mathfrak{g}_\xi)$, with $\Delta=\mathfrak{t}^{*}\backslash\{0\}$.}

In this section, we will study the Cartan decomposition of $\mathfrak{g}^{\Delta}$, when $\text{Card}(\Delta)= 7$, that is, when $\Delta = \Delta_{0}=\mathbf{t}^{*}\backslash\{0\}. $ We will prove that there not simple Lie 2-algebra of $10\leq \text{dim}_{K}(\mathfrak{g})\leq 16$ with  $MT(\mathfrak{g})=3$.

From now on, we will suppose that $(\mathfrak{g},[2])$ is a simple Lie $2$-algebra of toral rank 3, with dimension greater that or equal to 10, whose Cartan decomposition with respect to $\mathfrak{t}$ is: \[\mathfrak{g}=\mathfrak{t}\oplus\mathfrak{n}\oplus \mathfrak{g}_{\alpha}\oplus\mathfrak{g}_{\beta}\oplus \mathfrak{g}_{\gamma}\oplus \mathfrak{g}_{\alpha+\beta}\oplus \mathfrak{g}_{\alpha+\gamma}\oplus \mathfrak{g}_{\beta+\gamma}\oplus \mathfrak{g}_{\alpha+\beta+\gamma}.\]

Remember that we are considering that \[\text{dim}_{K}(\mathfrak{g}_{\alpha})\geq \text{dim}_{K}(\mathfrak{g}_{\beta})\geq \text{dim}_{K}(\mathfrak{g}_{\gamma})\geq \text{dim}_{K}(\mathfrak{g}_{\xi})> 0\quad \text{for}\quad \xi\notin \Delta\backslash\{\alpha,\beta,\alpha+\beta\}\]
(see Remark \ref{obs33}).

\medskip

 Next theorem will help us to study (classify) those algebras in whose decomposition of Cartan the root  spaces   have different dimensions, which will be fundamental to show our main result (Theorem \ref{theorem53}). 

\begin{theorem}\label{proposicion51}
Take $\alpha,\beta\in\Delta$. If there exists $e_\alpha\in \mathfrak{g}_\alpha$ such that \[e_\alpha^{[2]}:=t_\alpha+n_\alpha, \quad t_\alpha\in \mathfrak{t}\backslash\{0\},\quad n_\alpha\in \mathfrak{n} \quad \text{and}\quad\beta(t_\alpha)\neq 0,\] then $\mathfrak{g}_\beta$ is isomorphic to $\mathfrak{g}_{\alpha+\beta}$.
\end{theorem}
\begin{proof}
Take $e_\alpha\in \mathfrak{g}_\alpha$ and let $\text{ad}(e_\alpha):\mathfrak{g}_\beta\rightarrow \mathfrak{g}_{\alpha+\beta}$ be the adjoint mapping. We prove that $\text{ad}(e_\alpha)$ is injective. Take $e_\beta$, $f_\beta$ in $\mathfrak{g}_{\beta}$ such that $\text{ad}(e_\alpha)(e_\beta)=\text{ad}(e_\alpha)(f_\beta)$. Then $[e_\alpha,e_\beta]=[e_\alpha,f_\beta].$ Therefore \[[e_\alpha^{[2]},e_\beta]=[e_\alpha,[e_\alpha,e_\beta]]=[e_\alpha,[e_\alpha,f_\beta]]=[{e_\alpha}^{[2]},f_\beta].\] 
By hypothesis $[t_\alpha+n_\alpha,e_\beta]=[t_\alpha+n_\alpha,f_\beta]$, that is, $$ \beta(t_\alpha)e_\beta+\text{ad}(n_\alpha)(e_\beta)=\beta(t_\alpha)f_\beta+\text{ad}(n_\alpha)(f_\beta)$$  and so $\beta(t_\alpha)(e_\beta+f_\beta)=\text{ad}(n_\alpha)(e_\beta+f_\beta).$ Therefore, $\text{ad}(t_\alpha)(e_\beta+f_\beta)=\text{ad}(n_\alpha)(e_\beta+f_\beta)$. As $\text{ad}(n_\alpha)$ is nilpotent, there exists $m\in\mathbb{N}$ such that ${\text{ad}(n_\alpha)}^{m}=0.$ 
So ${\text{ad}(t_\alpha)}^{m}(e_\beta+f_\beta)=0$ and therefore $(\text{ad}(t_\alpha)(e_\beta+f_\beta))^{m}=0$. Then $\text{ad}(t_\alpha)(e_\beta+f_\beta)=0$ and so  $\beta(t_\alpha)(e_\beta+f_\beta)=0$.  As $\beta(t_\alpha)\neq 0$ we have $e_\beta=f_\beta$ and so $\text{ad}(e_\alpha)$  is injective.

Analogously we can prove that  $\text{ad}(e_\alpha):\mathfrak{g}_{\alpha+\beta}\rightarrow \mathfrak{g}_{\beta}$ is an injective linear transformation, since $(\alpha+\beta)(t_\alpha)=\alpha(t_\alpha)+\beta(t_\alpha)=0+\beta(t_\alpha)=\beta(t_\alpha)\neq0$. 

The rank-nullity formula implies that
\[\text{dim}_{K}( \mathfrak{g}_{\beta})=\text{dim}_{K}\text{Im}(\text{ad}(e_\alpha))\leq \text{dim}_{K}(\mathfrak{g}_{\alpha+\beta})\]  and 
\[\text{dim}_{K}(\mathfrak{g}_{\alpha+\beta})=\text{dim}_{K}\text{Im}(\text{ad}(e_\alpha))\leq \text{dim}_{K}(\mathfrak{g}_{\beta}).\] Therefore 
$\text{dim}_{K}(\mathfrak{g}_{\beta})=\text{dim}_{K}(\mathfrak{g}_{\alpha+\beta})$.
\end{proof}

\begin{notation} For $\mathfrak{j}\in \{\mathfrak{g},\mathfrak{t},\mathfrak{n},\mathfrak{\mathfrak{g}_\alpha},\mathfrak{\mathfrak{g}_\beta},\mathfrak{g}_{\gamma},\mathfrak{g}_{\alpha+\beta},\mathfrak{g}_{\alpha+\gamma},\mathfrak{g}_{\beta+\gamma},\mathfrak{g}_{\alpha+\beta+\gamma}\}$, we set    $$\text{dim}_{K}(\mathfrak{j})=\text{d}(\mathfrak{j}).$$  Consider
\[\mathfrak{P}=(\text{d}(\mathfrak{g}):\text{d}(\mathfrak{t}),\text{d}(\mathfrak{n}),\text{d}(\mathfrak{g}_\alpha),\text{d}(\mathfrak{g}_\beta),\text{d}(\mathfrak{g}_\gamma),\text{d}(\mathfrak{g}_{\alpha+\beta}),\text{d}(\mathfrak{g}_{\alpha+\gamma}),\text{d}(\mathfrak{g}_{\beta+\gamma}),\text{d}(\mathfrak{g}_{\alpha+\beta+\gamma})).\]
\end{notation}

\medskip


Next,  fixing the dimension of $(\mathfrak{g},[2])$, we will study its structure when $\Delta=\mathfrak{t}^{*}\backslash\{0\}.$ We have the followings possibilities: 
\begin{enumerate}[1.]
    \item If $\text{dim}_{K}(\mathfrak{g})=10$, we have $\mathfrak{P}=(10: 3,0,1,1,1,1,1,1,1)$. Thus $\mathfrak{I}=\underset{\xi\in\Delta}\oplus \mathfrak{g}_{\xi} $ is an ideal of $\mathfrak{g}$, so $\mathfrak{g} $ is not simple. A contradiction.
    \item If $ \text{dim}_K (\mathfrak{g}) = 11 $, we have the following cases:
    \begin{enumerate}[i.]
        \item $\mathfrak{P}=(11:3,0,2,1,1,1,1,1,1)$. In this case, as  $\mathfrak{t}\subseteq\sum_{\xi\in\Delta}[\mathfrak{g}_\xi,\mathfrak{g}_\xi]\subseteq[\mathfrak{g}_\alpha,\mathfrak{g}_\alpha],$ then the $\text{dim}_{K}(\mathfrak{t})\leq1$ ( A contradiction). 
        \item $\mathfrak{P}=(11:3,1,1,1,1,1,1,1,1)$. Then $\mathfrak{I}:=\mathfrak{n}\oplus(\underset{\xi\in\Delta}\oplus \mathfrak{g}_\xi)$ is an ideal of $\mathfrak{g}$. A contradiction.
    \end{enumerate}
    \item If $\text{dim}_{K}(\mathfrak{g})=12$, we have:
    \begin{enumerate}[i.]
        \item $\mathfrak{P}=(12:3,0,3,1,1,1,1,1,1)$. If $e_\alpha\in[\mathfrak{g}_\alpha,\mathfrak{g}_\alpha]$, then there exist $\delta_1,\delta_2,\delta_3\in K$ such that $e_\alpha=\delta_1t_1+\delta_2t_2+\delta_3t_3$. Thus $$ 0=[e_\alpha,e_\alpha]=\delta_1\alpha(t_1)e_\alpha+\delta_2\alpha(t_2)e_\alpha+\delta_3\alpha(t_3)e_\alpha =\delta_1e_\alpha.$$ Therefore $\delta_1=0$ and so $e_\alpha=\delta_2t_2+\delta_3t_3$, that is, $e_\alpha\in\left\langle t_2,t_3\right\rangle$. Hence  $[\mathfrak{g}_\alpha,\mathfrak{g}_\alpha]\subseteq\left\langle t_2,t_3 \right\rangle.$
        Since $\mathfrak{t}\subseteq\sum_{\xi\in\Delta}[\mathfrak{g}_\xi,\mathfrak{g}_\xi]\subseteq[\mathfrak{g}_\alpha,\mathfrak{g}_\alpha]\subseteq\left\langle t_2,t_3\right\rangle$, we have  $\text{dim}_{K}(\mathfrak{t})\leq2$, which is a  contradiction.
        \item $\mathfrak{P}=(12:3,0,2,2,1,1,1,1,1)$. Then $\mathfrak{t}\subseteq\sum_{\xi\in\Delta}[\mathfrak{g}_\xi,\mathfrak{g}_\xi]\subseteq[\mathfrak{g}_\alpha,\mathfrak{g}_\alpha]+[\mathfrak{g}_\beta,\mathfrak{g}_\beta]$ and so $\text{dim}_{K}(\mathfrak{t})\leq \text{dim}_{K}([\mathfrak{g}_\alpha,\mathfrak{g}_\alpha])+\text{dim}_{K}([\mathfrak{g}_\beta,\mathfrak{g}_\beta])\leq 1+1=2$.
        \item $\mathfrak{P}=(12:3,1,2,1,1,1,1,1,1)$.  We have
        $\mathfrak{h}\subseteq\sum_{\xi\in\Delta}[\mathfrak{g}_\xi,\mathfrak{g}_\xi]=[\mathfrak{g}_\alpha,\mathfrak{g}_\alpha]$ and since $\text{dim}_{K}(\mathfrak{g}_\alpha)=2$, it follows that $\text{dim}_{K}([\mathfrak{g}_\alpha,\mathfrak{g}_\alpha])\leq1$. So,
       $\text{dim}_{K}(\mathfrak{h})\leq1$. A contradiction.
       \item $\mathfrak{P}=(12:3,2,1,1,1,1,1,1,1)$. We have $\mathfrak{I}:=\mathfrak{n}\oplus\sum_{\xi\in\Delta}\oplus \mathfrak{g}_\xi$ is an  ideal of $\mathfrak{g}$. A contradiction.
 \end{enumerate}
    \item If $\text{dim}_{K}(\mathfrak{g})=13$, then we have the following structures for $\mathfrak{g}$:
  \begin{enumerate}[i.]
       \item $\mathfrak{P}=(13:3,0,4,1,1,1,1,1,1)$.
       \item $\mathfrak{P}=(13:3,0,3,2,1,1,1,1,1)$.
      \item $\mathfrak{P}=(13:3,0,2,2,2,1,1,1,1)$.\\
For these three cases we have:  
let $e_{\beta}\in\mathfrak{g}_{\beta}$ be an  element of the base of $\mathfrak{g}_{\beta}$. Then $e_{\beta}^{[2]}\in\mathfrak{t}$ and so there exist $\delta_1,\delta_2,\delta_3\in K$ such that $e_\beta^{[2]}=\delta_1t_1+\delta_2t_2+\delta_3t_3$. Since $\beta(e_\beta^{[2]})=0$, we have $\delta_2=0$. Thus $t_\beta=\delta_1t_1+\delta_3t_3$ for some  $\delta_1,\delta_3\in K$ and so
$\mathfrak{g}_\beta^{[2]}\subseteq\left\langle t_1,t_3\right\rangle$. If we choose $\delta_1\neq0$,  then $\alpha(t_{\beta})=\delta_1\neq0$. By Theorem \ref{proposicion51} we have $\mathfrak{g}_\alpha\simeq \mathfrak{g}_{\alpha+\beta}$. However, this is impossible because  $\text{dim}_{K}(\mathfrak{g}_\alpha)\neq \text{dim}_{K}(\mathfrak{g}_{\alpha+\beta})=1$.
\item $\mathfrak{P}=(13:3,1,1,1,1,1,1,1).$ Thus  $\mathfrak{I}:=\mathfrak{n}\oplus(\underset{\xi\in\Delta}\oplus \mathfrak{g}_\xi)$ is an ideal of $\mathfrak{g}$, which is absurd.  
\item $\mathfrak{P}=(13:3,2,2,1,1,1,1,1,1)$. We have $\mathfrak{h}\subseteq\sum_{\xi\in\Delta}[\mathfrak{g}_\xi,\mathfrak{g}_\xi]=[\mathfrak{g}_\alpha,\mathfrak{g}_\alpha]$. Thus $\text{dim}_{K}(\mathfrak{h})\leq1$. A contradiction.
\item $\mathfrak{P}=(13:3,1,3,1,1,1,1,1,1)$. We have \[\mathfrak{h}\subseteq\sum_{\xi\in\Delta}[\mathfrak{g}_\xi,\mathfrak{g}_\xi]=[\mathfrak{g}_\alpha,\mathfrak{g}_\alpha]\subseteq\left\langle t_2,t_3\right\rangle \] and so $\text{dim}_{K}(\mathfrak{h})\leq2$. A contradiction.
\item $\mathfrak{P}=(13:3,1,2,2,1,1,1,1,1)$. Thus $\mathfrak{h}\subseteq\sum_{\xi\in\Delta}[\mathfrak{g}_\xi,\mathfrak{g}_\xi]=[\mathfrak{g}_\alpha,\mathfrak{g}_\alpha]+[\mathfrak{g}_\beta,\mathfrak{g}_\beta].$ Hence $\text{dim}_{K}(\mathfrak{h})\leq2$. A contradiction. 
\end{enumerate}
\item If $\text{dim}_{K}(\mathfrak{g})=14$. We have the following possibilities for $\mathfrak{P}$:
\begin{multicols}{2}
\begin{enumerate}[i.]
\item $(14:3,0,5,1,1,1,1,1,1)$ 
\item $(14:3,0,4,2,1,1,1,1,1)$
\item $(14:3,0,3,3,1,1,1,1,1)$
\item $(14:3,0,3,2,2,1,1,1,1)$
\item $(14:3,0,2,2,2,2,1,1,1)$
\item $(14:3,1,2,2,2,1,1,1,1)$
\item $(14:3,1,3,2,1,1,1,1,1)$
\item $(14:3,1,4,1,1,1,1,1,1)$
\item $(14:3,2,2,2,1,1,1,1,1)$
\item $(14:3,2,3,1,1,1,1,1,1)$
\item $(14:3,3,2,1,1,1,1,1,1)$
\item $(14:3,4,1,1,1,1,1,1,1)$
\end{enumerate}
\end{multicols}
 For the case  xii, we have that $\mathfrak{I}:=\mathfrak{n}\oplus(\underset{\xi\in\Delta}\oplus \mathfrak{g}_\xi)$ is an ideal of $\mathfrak{g}$, which is absurd.
 Now, let $e_\gamma$ be an element of the base of $\mathfrak{g}_\gamma$. As $e_\gamma^{[2]}\in \mathfrak{g}_\gamma^{[2]}\subseteq \mathfrak{h}=\mathfrak{t}\oplus\mathfrak{n}$, then there are  $t_\gamma\in \mathfrak{t}$ and $n_\gamma\in \mathfrak{n}$ such that  $e_\gamma^{[2]}=t_\gamma+n_\gamma$. As $t_\gamma\in \mathfrak{t}$, there exist $\delta_1$, $\delta_2$, $\delta_3$ in $K$, such that $t_\gamma=\delta_1t_1+\delta_2t_2+\delta_3t_3$. On the other hand, $0=[e_\gamma,e_\gamma^{[2]}]=\gamma(t_\gamma)e_\gamma+[e_\gamma,n_\gamma]$, then $\text{ad}(n_\gamma)(e_\gamma)=\gamma(t_\gamma)e_\gamma$. As $n_\gamma$ is nilpotent then $\text{ad}(n_\gamma)$ is nilpotent and therefore $\gamma(t_\gamma)=\delta_3=0$. So, $e_\gamma^{[2]}= \delta_1t_1+\delta_2t_2 + n_\gamma$. Choosing  $\delta_3\neq0$ we have $t_\gamma\neq0$ and $\alpha(t_\gamma)=\delta_1\neq0$, then by Theorem  \ref{proposicion51}, $\mathfrak{g}_\alpha\simeq \mathfrak{g}_{\alpha+\gamma}$. This fact shows that the remaining cases are impossible because $\text{dim}_{K}(\mathfrak{g}_\alpha)\neq \text{dim}_{K}(\mathfrak{g}_{\alpha+\gamma})=1$.
\item If $\text{dim}_{K}(\mathfrak{g})=15$, we have the following cases:
\begin{multicols}{2}
\begin{enumerate}[i.]
\item $(15:3,5,1,1,1,1,1,1,1)$
\item $(15:3,4,2,1,1,1,1,1,1)$
\item $(15:3,3,2,2,1,1,1,1,1)$
\item $(15:3,2,3,2,1,1,1,1,1)$
\item $(15:3,2,2,2,2,1,1,1,1)$
\item $(15:3,1,3,2,2,1,1,1,1)$
\item $(15:3,0,6,1,1,1,1,1,1)$
\item $(15:3,0,5,2,1,1,1,1,1)$
\item $(15:3,0,4,3,1,1,1,1,1)$
\item $(15:3,0,3,3,2,1,1,1,1)$
\item $(15:3,0,3,2,2,2,1,1,1)$
\item $(15:3,0,2,2,2,2,2,1,1)$
\item $(15:3,0,2,2,2,2,2,1,1)$
\end{enumerate}
\end{multicols}
For the case i. we have $\mathfrak{I}:=\mathfrak{n}\oplus(\underset{\xi\in\Delta}\oplus \mathfrak{g}_\xi)$ is an ideal of $\mathfrak{g}$. A contradiction. As in the above cases, let $e_\beta$ be an element of the  basis of $\mathfrak{g}_\beta$. Since $e_\beta^{[2]}\in \mathfrak{g}_\beta^{[2]}\subseteq \mathfrak{h}=\mathfrak{t}\oplus\mathfrak{n}$,   there exist $t_\beta\in \mathfrak{t}$ and $n_\beta\in \mathfrak{n}$ such that  $e_\beta^{[2]}=t_\beta+n_\beta$. As $t_\beta\in\mathfrak{t}$, there are $\delta_1$, $\delta_2$, $\delta_3$ in $K$, such that $t_\beta=\delta_1t_1+\delta_2t_2+\delta_3t_3$. On the other hand, $0=[e_\beta,e_\beta^{[2]}]=\beta(t_\beta)e_\beta+[e_\beta,n_\beta]$ and thus  $\text{ad}(n_\beta)(e_\beta)=\beta(t_\beta)e_\beta$. Since $n_\beta$ is nilpotent we have $\text{ad}(n_\beta)$ is nilpotent and therefore $\beta(t_\beta)=\delta_2=0$. So, $e_\beta^{[2]}= \delta_1t_1+\delta_3t_3 + n_\beta$. Choosing  $\delta_1\neq0$, we have $t_\beta\neq0$ and $\alpha(t_\beta)=\delta_1\neq0$. Thus, by Theorem \ref{proposicion51},  we have $\mathfrak{g}_\alpha\simeq \mathfrak{}_{\alpha+\beta}$. Therefore,   the cases from ii. to  xii.  are all impossible, because   $\text{dim}_{K}(\mathfrak{g}_\alpha)\neq \text{dim}_{K}(\mathfrak{g}_{\alpha+\beta})=1$.

On the other hand, choosing $\delta_1\neq\delta_3$,   we have  \[(\alpha+\gamma)(t_\beta)=(\alpha+\gamma)(\delta_1t_1+\delta_3t_3)=\delta_1+\delta_3\neq0.\] Therefore, by Theorem \ref{proposicion51}, we have $\mathfrak{g}_{\alpha+\gamma}\simeq \mathfrak{g}_{\alpha+\beta+\gamma}$. However \[2=\text{dim}_{K}(\mathfrak{g}_{\alpha+\gamma})\neq \text{dim}_{K}(\mathfrak{g}_{\alpha+\beta+\gamma})=1,\] then the case xiii.  is not  possible.
\item If  $\text{dim}_{K}(\mathfrak{g})=16$, we have the following cases:
\begin{multicols}{2}
\begin{enumerate}[i.]
\item $(16:3,6,1,1,1,1,1,1,1)$
\item $(16:3,5,2,1,1,1,1,1,1)$
\item $(16:3,4,3,1,1,1,1,1,1)$
\item $(16:3,4,2,2,1,1,1,1,1)$
\item $(16:3,3,4,1,1,1,1,1,1)$
\item $(16:3,3,3,2,1,1,1,1,1)$
\item $(16:3,3,2,2,2,1,1,1,1)$
\item $(16:3,2,5,1,1,1,1,1,1)$
\item $(16:3,2,4,2,1,1,1,1,1)$
\item $(16:3,2,3,3,1,1,1,1,1)$
\item $(16:3,2,3,2,2,1,1,1,1)$
\item $(16:3,2,2,2,2,2,1,1,1)$
\item $(16:3,1,6,1,1,1,1,1,1)$
\item $(16:3,1,5,2,1,1,1,1,1)$
\item $(16:3,1,4,3,1,1,1,1,1)$
\item $(16:3,1,4,2,2,1,1,1,1)$
\item $(16:3,1,3,3,2,1,1,1,1)$
\item $(16:3,1,2,2,2,2,2,1,1)$
\item $(16:3,0,7,1,1,1,1,1,1)$
\item $(16:3,0,6,2,1,1,1,1,1)$
\item $(16:3,0,5,2,2,1,1,1,1)$
\item $(16:3,0,4,3,2,1,1,1,1)$
\item $(16:3,0,4,2,2,2,1,1,1)$
\item $(16:3,0,4,2,2,2,1,1,1)$
\item $(16:3,0,4,2,2,2,1,1,1)$
\item $(16:3,0,3,3,2,2,1,1,1)$
\item $(16:3,0,3,2,2,2,2,1,1)$
\item $(16:3,0,2,2,2,2,2,2,1)$
\end{enumerate}
\end{multicols}
\end{enumerate}
For the  case i.,  we have $\mathfrak{I}:=\mathfrak{n}\oplus(\underset{\xi\in\Delta}\oplus \mathfrak{g}_\xi)$ is an ideal of $\mathfrak{g}$, which is a contradiction.   Let $e_\beta$ be an element of the basis of $\mathfrak{g}_\beta$. As $e_\beta^{[2]}\in \mathfrak{g}_\beta^{[2]}\subseteq \mathfrak{h}=\mathfrak{t}\oplus\mathfrak{n}$, then there are $t_\beta\in \mathfrak{t}$ and $n_\beta\in\mathfrak{n}$ such that  $e_\beta^{[2]}=t_\beta+n_\beta$. As $t_\beta\in \mathfrak{t}$, there exist $\delta_1$, $\delta_2$, $\delta_3$ in $K$, such that $t_\beta=\delta_1t_1+\delta_2t_2+\delta_3t_3$. On the other hand, $0=[e_\beta,e_\beta^{[2]}]=\beta(t_\beta)e_\beta+[e_\beta,n_\beta]$ and then $\text{ad}(n_\beta)(e_\beta)=\beta(t_\beta)e_\beta$. Since $n_\beta$ is nilpotent,   $\text{ad}(n_\beta)$ is nilpotent and therefore $\beta(t_\beta)=\delta_2=0$. So $e_\beta^{[2]}= \delta_1t_1+\delta_3t_3 + n_\beta$. Choosing $\delta_1\neq0$ we have $t_\beta\neq0$ and $\alpha(t_\beta)=\delta_1\neq0$. By Theorem \ref{proposicion51}, we have $\mathfrak{g}_\alpha\cong \mathfrak{g}_{\alpha+\beta}$. Therefore the cases from i. to xxvii.,    except xii. and xix., are impossible, because $\text{dim}_{K}(\mathfrak{g}_{\alpha})\neq \text{dim}_{K}(\mathfrak{g}_{\alpha+\beta})$.
On the other hand, choosing $\delta_3\neq0$, we have  $\gamma(t_\beta)=\delta_3\neq0$. Thus, by Theorem \ref{proposicion51}, we have $\mathfrak{g}_\gamma\simeq \mathfrak{g}_{\beta+\gamma}$, but $2=\text{dim}_{K}(\mathfrak{g}_\gamma)\neq \text{dim}_{K}(\mathfrak{g}_{\beta+\gamma})=1$. Hence  the cases xii. and xix. are not possible. 
In the case  xxviii. choosing $\delta_1\neq\delta_3$, we have $(\alpha+\gamma)(t_\beta)=\delta_1+\delta_3\neq0$. Therefore, by Theorem \ref{proposicion51} we have $\mathfrak{g}_{\alpha+\gamma}\simeq \mathfrak{g}_{\alpha+\beta+\gamma}$, which is absurd, since $2=\text{dim}_{K}(\mathfrak{g}_{\alpha+\gamma})\neq \text{dim}_{K}(\mathfrak{g}_{\alpha+\beta+\gamma})=1$ and hence this case is not possible.

\medskip

It follows from the above facts that:

\begin{proposition}\label{ultimoprop}
There are no simple Lie $2$-algebras of dimension between 10 and 16,  and with  toral rank $3$.
\end{proposition}

We finish this work presenting our main result, which follows from 
  Propositions \ref{payares},  \ref{payarin} and \ref{ultimoprop}.

\begin{theorem}\label{theorem53}
There are no simple Lie $2$-algebras of dimension less than or equal to $16$, and toral rank $3$.
\end{theorem}

\end{document}